\documentclass{article}
\usepackage{amsmath,amsthm, amssymb}
\usepackage{enumitem}
\usepackage[dvipsnames]{xcolor}
\usepackage[centering]{geometry}
\usepackage{mathrsfs}
\usepackage{titling}

\usepackage[color=yellow!30, linecolor=orange, colorinlistoftodos]{todonotes}
\usepackage[breakable]{tcolorbox}

\usepackage{tikz}
\usepackage{tikzrput}
\usetikzlibrary{decorations,decorations.markings}
\usetikzlibrary{decorations.pathreplacing,shapes.misc}

\usepackage{pgfplots}
\pgfplotsset{compat=1.18}
\usetikzlibrary{math} 
\usetikzlibrary{intersections}  

\usepackage{fancyhdr}
\usepackage{soul}  

\usepackage[hypertexnames=false, breaklinks,
	colorlinks=true, 
	citecolor=red!70!black,
	linkcolor=blue,
	anchorcolor=cyan,
	urlcolor=green!50!black,
	pagebackref=false,
	]{hyperref}

\usepackage{comment}

\providecommand{\MR}{\relax\ifhmode\unskip\space\fi MR }

\numberwithin{equation}{section}

\colorlet{mycyan}{cyan!20}
\colorlet{mymagenta}{magenta!30}
\colorlet{myorange}{orange!40}

\newtheorem{theorem}{Theorem}[section]
\newtheorem{lemma}[theorem]{Lemma}

\theoremstyle{definition}
\newtheorem{definition}[theorem]{Definition}

\newtheorem{remark}[theorem]{Remark}
\newtheorem*{remark*}{Remark}
\newtheorem{corollary}[theorem]{Corollary}
\newtheorem{proposition}[theorem]{Proposition}
\newcommand{\wt}[1]{\widetilde{#1}}
\newcommand{\wh}[1]{\widehat{#1}}

\newcommand{\kip}{[\,\cdot\, , \cdot\,]}

\newcommand{\cK}{{\mathcal K}}
\newcommand{\cH}{{\mathcal H}}

\newcommand{\R}{\mathbb{R}}
\newcommand{\C}{\mathbb{C}}
\newcommand{\N}{\mathbb{N}}
\newcommand{\Z}{\mathbb{Z}}
\newcommand{\mD}{\mathcal{D}}

\DeclareMathOperator{\sign}{sign}

\DeclareMathOperator{\linspan}{span}

\newcommand{\I}{\mathrm i} 
\newcommand{\rd}{\mathrm d} 
 
\newcommand{\scalar}[2]{\langle#1,\, #2\rangle}

\newcommand{\kscalar}[2]{[#1,\, #2]}

\newcommand{\define}[1]{\textit{#1}}
\newcommand{\pointspec}{\sigma_{\mathrm p}}

\newcommand{\DiffInt}{\Aj_0}  
\newcommand{\Aj}[1][j]{A^{(#1)}}  

\newcommand{\AK}{A}
\newcommand{\AKone}{B}

\newcommand{\SK}{S}
\newcommand{\TK}{T}
\newcommand{\MK}{M}
\newcommand{\gammaK}{\gamma}

\newcommand{\eigf}[1]{u_{#1}}  

\DeclareMathOperator{\ran}{ran}
\DeclareMathOperator{\dom}{\mD}

\title{1-D  Schr\"odinger operator on a star graph with nondefinite weight function}

\author{Edison Leguizam\'on, Carsten Trunk, Mitsuru Wilson, and Monika Winklmeier\\ 
\\
Dedicated to the memory of Heinz Langer, whom we deeply admired.\\
His legacy will live on in our practices.\\
}

\begin{document}

\maketitle
\begin{abstract}
We investigate the indefinite Kirchhoff Laplacian on a star graph \(G\) with \(n = n_+ + n_-\) edges of unit length, where the differential operator acts as \(-\frac{d^2}{dx^2}\) on \(n_+\) edges and \(\frac{d^2}{dx^2}\) on \(n_-\) edges. The operator is subject to Dirichlet conditions at the outer vertices and Kirchhoff conditions at the central vertex. We establish that this operator is similar to a self-adjoint operator in the Hilbert space \(L^2(G)\) and that its eigenfunctions constitute a Riesz basis. Additionally, we provide a complete characterization of its point spectrum.
\end{abstract}

\section{Introduction}
\noindent
The study of differential operators on metric graphs has attracted considerable attention over the past decades. Much of the existing literature has focused on selfadjoint Schr\"{o}dinger operators (see the monographs \cite{BerkolaikoKuchment, BEH08, K24} and the seminal paper of V.\ Kostrykin and R.\ Schrader \cite{KS99}). 

The present work is motivated by the growing interest in spectral theory on network structures and by the renewed relevance of non-selfadjoint operators in quantum mechanics (see, e.g.,\ \cite{HKS15}). There have been recent attempts to formulate “quasi-Hermitian quantum mechanics”, where physical observables are represented by non-selfadjoint operators $T$ satisfying the quasi-selfadjointness relation
\begin{equation}
\label{Theta}
	T^* = \Theta T \Theta^{-1}
\end{equation}
for a bounded, boundedly invertible operator $\Theta$ (see, e.g., \cite{SGH92}). In other words, $T$ is similar to a selfadjoint operator. 

In \cite{CurgusNajman1995}, B.~{\'C}urgus and B.~Najman considered the differential operator $\frac{\rd^2}{\rd x^2}$ on $\R^-$ and $-\frac{\rd^2}{\rd x^2}$ on $\R^+$ in $L^2(\mathbb R)$ defined on the Sobolev space $H^2(\R)$. This operator serves as one of the pioneering examples of a simple differential operator that, while not selfadjoint in the standard $L^2$-space over $\R$, is nevertheless spectrally well understood as it satisfies \eqref{Theta}. From the view point of graphs, this operator is the (indefinite) Kirchhoff Laplacian  defined on a star graph with two infinite edges (namely $\mathbb R^-$ and $\mathbb R^+$) and a vertex in zero. We refer to \cite{CurgusLanger1989} for a broader class of such differential operators and to \cite{F24} for an overview, see also \cite{CDT24}. 

\smallskip
The present paper can be viewed as a generalization of these results to graphs. It follows the way paved by the seminal paper of B.~{\'C}urgus and H.\ Langer \cite{CurgusLanger1989} and uses Krein space techniques developed by the late M.G.\ Krein and H.\ Langer, we mention here only  \cite{IKL,K70,L65,La}.
These techniques were originally used to develop a spectral theory of (monic) selfadjoint operator pencils \cite{L67,L73,L74,L75}.

Specifically, we study here the spectral properties of the indefinite Kirchhoff Laplacian on a star graph $G$ with $n$ edges $I_1,\, \dots,\, I_n$, each of unit length, and all joined at the central vertex $0$. 
We will work in the space of square integrable functions on $G$ given by
\begin{equation}
\label{eq:hilbert_graph}
	L^2(G) := \bigoplus_{j=1}^n L^2(I_j) = 
    \bigoplus_{j=1}^n L^2((0,1)). 
\end{equation}
We denote $f_j := f|_{I_j}$ for a function $f=(f_1,\, \dots,\, f_n)\in L^2(G)$.
Our goal is to investigate the spectral properties of the indefinite Kirchhoff Laplacian $\AKone$ on the star graph $G$ with Dirichlet conditions at the outer vertices and a Kirchhoff condition at the central vertex in $L^2(G)$.
Let us denote 
\begin{equation}
   \label{eq:hilbert_graph.dirichlet}
   \widehat H^1_D(G) := \{ f\in L^2(G) : 
   f_j\in H^1(I_j) \text{ for all } j \text{ and }
   f_1(1) = \cdots = f_n(1) = 0 \}.
\end{equation}

The corresponding operator $\AKone$ is then defined as
\begin{align}
\label{DerAuserwaehlte}
   (\AKone f)_j = 
   \begin{cases}
      -f_j''\quad &\text{if }\ 1 \le j \le n_+,\\
      f_j''\quad &\text{if }\  n_+ + 1 \le j \le n,
   \end{cases}
\end{align}
on the domain
\begin{align}
\label{ElDominio}
   \mD(B) = \left\{ f\in \widehat H_D^1(G)\ : \
   f_j\in H^2(I_j) \text{ for all } j \text{ and }
   f_1(0) = \dots = f_n(0), ~ \sum_{j=1}^n f_j'(0) = 0
\right\}.
\end{align}
Using boundary triple techniques and the associated Weyl function (see the next section below) we show in Section \ref{spectrum} that $B$ is a positive selfadjoint operator in the sense of Krein spaces and give a complete description of its eigenvalues.
In particular one can recover from the spectrum parts of the topology of the graph $G$, namely the ratio between $n_+$ and $n_-$. Finally, as our second main result, we show that the indefinite Kirchhoff Laplacian $B$ satisfies \eqref{Theta}. In other words, $B$ is similar to a selfadjoint operator.

The last result was already shown in \cite{CW} with different techniques in a more general setting. We mention also the recent paper by S.\ Nicaise \cite{Nic}, where a similar problem is considered, but the sign change is put between the first and second derivative in the definition of $B$, which leads to a somehow different model. The boundary conditions at the central vertex are different from ours and make the corresponding Kirchhoff Laplacian in \cite{Nic} then selfadjoint in a Hilbert space.

\medskip
{\bf Notations.}
By $\mathbb{C}$ we denote the set of complex numbers and by $\mathbb{R}$ the set of real numbers. Here $\mathbb N = \{1, 2, 3, \dots\}$ stands for the natural non-zero numbers and $\mathbb Z$ the integers. All operators are closed densely defined linear operators. For such an operator $T$, we use the common notation $\rho(T)$, $\sigma(T)$, $\pointspec(T)$, $\ran T$ and $\mD(T)$ for its resolvent set, spectrum, point spectrum, range and domain, respectively. 
Moreover, $H^1((0,1))$ stands for the first Sobolev space, that is, the subset of all functions $f\in L^2((0,1))$ such that $f$ has a first weak derivative and $f$ and $f'$ have a finite $L^2$-norm. 
Note that any $f\in H^1((0,1))$ is continuous in $(0,1)$ with a continuous extension to the endpoints $0$ and $1$.
Analogously, $H^2((0,1))$ stands for the second Sobolev space which is the subset of all functions $f\in H^1((0,1))$ such that $f'\in H^1((0,1))$. 
By $H^1_0((0,1))$ we denote all functions $f\in H^1((0,1))$ with $f(0)=f(1)=0$ and by $\mathcal C_0^\infty((0,1))$ the space of all infinitely often differentiable functions on the interval $[0,1]$ with compact support in $(0,1)$.

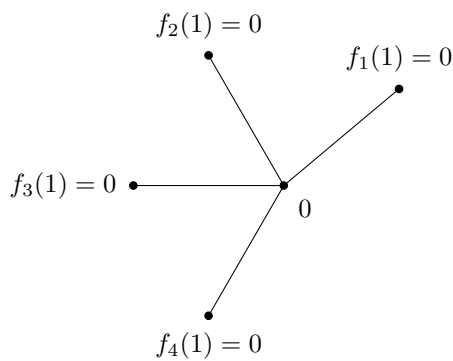
\begin{figure}[h]
   \begin{center}
      \begin{tikzpicture}[transform shape, baseline=(D)] 
	 \coordinate (O) at (0,0);
	 \coordinate (A) at (40:2);
	 \coordinate (B) at (120:2);
	 \coordinate (C) at (180:2);
	 \coordinate (D) at (240:2);
	 \coordinate (E) at (330:2);
	 \coordinate (F) at (310:2);
	 \foreach \x in {A,B,C,D} \draw (O) -- (\x);
	 \foreach \x in {A,B,C,D,O}  \draw[fill] (\x) circle (0.05cm);
	 \node at (A) [above=.1] {$f_1(1) = 0$};
	 \node at (B) [above=.1] {$f_2(1) = 0$};
	 \node at (C) [left=.1] {$f_3(1) = 0$};
	 \node at (D) [below=.1] {$f_4(1) = 0$};
	 \node at (O) [below right=.1] {$0$};
      \end{tikzpicture}
   \end{center}
   \caption{A star graph with 4 edges, 4 outer vertices and one central vertex (identified with zero) indicating that Dirichlet boundary conditions prevail at the four outer vertices.}
   \label{Bischleben}
\end{figure}

\section{Preliminaries: Non-negative operators in Krein spaces}
\label{sec:interval}

\noindent
A Krein space is a complex vector space $\cK$ equipped with a sesquilinear form $\kip$ such that there exist subspaces $\cH_+$ and $\cH_-$ of $\cK$  such that both $(\cH_-,\, -\kip)$ and $(\cH_+,\, \kip)$ are Hilbert spaces. 
    Moreover, $\cK$ admits the decomposition
\begin{equation}
\label{Bitterfeld}
	\cK = \cH_+[\dot{+}]   \cH_-,
\end{equation}
as a direct sum and orthogonal with respect to $\kip$. This decomposition is called a \define{fundamental decomposition} of the Krein space $\cK$. Let $P_+$ and $P_-$ be projections associated with the direct sum $\cK = \cH_+ \dot{+} \cH_-$. The operator 
\begin{equation*}
	J := P_+ - P_-
\end{equation*}
is called a \define{fundamental symmetry} of the Krein space $\cK$.  Equipped with the inner product  
\begin{equation}
\label{Halle}
	\scalar{x}{y} = [Jx,y], \quad  x,y \in \cK,
\end{equation}
the space $\cK$ becomes a Hilbert space.  Throughout, the topological structure of the Krein space $\cK$ is given with respect to the topology induced by this Hilbert space inner product. For the general theory of Krein spaces and operators acting in them, we refer to the monographs \cite{AI, B74, G24}.

Let  $A$ be a linear operator in a Krein space $\bigl(\cK,\kip\bigr)$ with a dense domain $\mD(A)$. The {\em adjoint} of $A$ with respect to the inner product $\kip$ is denoted by $A^+$. The operator $A$ is called {\em symmetric} in $\bigl(\cK,\kip\bigr)$ if $A^+$ is an extension of $A$ and $A$ is called {\em selfadjoint} in $\bigl(\cK,\kip\bigr)$ if $A = A^+$.  The operator $A$ is called {\em non-negative}  ({\em positive}) in $\bigl(\cK,\kip\bigr)$ if $[Af,f]\ge 0$
(resp.\ $[Af,f]>0$)
for all $f\in\mD(A)$ (for all $f\in\mD(A)\setminus \{0\}$, respectively). Let $A$ be a selfadjoint  non-negative operator with a nonempty resolvent set $\rho(A)$. In this case, its spectrum is real (see~\cite{Ando} or \cite[Theorem~II.3.1]{La}). Additionally, the operator $JA$ is non-negative in the Hilbert space $\bigl(\cK,\scalar{\cdot}{\cdot}\bigr)$ and has a well-defined square root $(JA)^{\frac{1}{2}}$.

The decomposition of the underlying space $\cK$ in \eqref{Bitterfeld} is not unique, consequently neither is the inner product in \eqref{Halle}. However, any two Banach space norms such that $\kip$ is continuous are equivalent (see, e.g., \cite{La}). This naturally leads to the following question: When is a given selfadjoint operator $A=A^+$ in a Krein space $\bigl(\cK,\kip\bigr)$ selfadjoint with respect to one of the possible Hilbert space inner products defined by \eqref{Bitterfeld} and \eqref{Halle}? This question can be formulated more generally: When is a given selfadjoint operator $A=A^+$ in a Krein space $\bigl(\cK,\kip\bigr)$ similar to a selfadjoint operator in a Hilbert space? 
Recall that $A$ is similar to a densely defined operator $\AKone$ if 
\begin{equation*}
	 A=VBV^{-1}
\end{equation*}
for some bounded linear operator $V$ with a bounded everywhere-defined inverse. 
This is one of the big questions in operator theory in Krein spaces, with many significant positive results; for an overview, we mention \cite{F24}, see also \cite{CDT24}. For our purposes, it suffices to study non-negative operators $A$ with $0\in\rho(A)$. We recall a result from Branko {\'C}urgus \cite{Curgus1985} in the version presented in \cite[Theorem 2.1]{CN94}.

\begin{theorem}
\label{Goettingen}
Let $A$ be a non-negative operator in a Krein space $\bigl(\cK,\kip\bigr)$ with $0\in\rho(A)$. Then, the following two statements are equivalent.

\begin{enumerate}[label={\upshape(\roman*)}]
    \item\label{Goettingen:i}
    There exists a bounded operator $W$, positive with respect to the Krein space inner product $\kip$ with a bounded everywhere defined inverse such that
    \begin{equation*}
	W\mD\left((JA)^{\frac{1}{2}}\right)\subset \mD\left((JA)^{\frac{1}{2}}\right). 
    \end{equation*}
		
    \item\label{Goettingen:ii}
    The operator $A$ is similar to a selfadjoint operator in the Hilbert space $\bigl(\cK,\scalar{\cdot}{\cdot}\bigr)$, where $\scalar{\cdot}{\cdot}$ is given by \eqref{Halle}.
\end{enumerate}
If the spectrum of $A$ is discrete, the above statements \ref{Goettingen:i} and \ref{Goettingen:ii} are equivalent to the following statement.
\begin{enumerate}[label={\upshape(\roman*)}]
\setcounter{enumi}{2}
    \item\label{Goettingen:iii}
    There exists an Riesz basis in $\cK$ consisting of eigenfunctions of $A$.
\end{enumerate}
\end{theorem}
The aforementioned Riesz basis is an orthonormal basis with respect to some Hilbert space inner product, which is equivalent to $\scalar{\cdot}{\cdot}$ in \eqref{Halle} \cite[Chapter VI \S 2]{GK}. 
\smallskip

We conclude this section by recalling some well-known facts from extension theory. 
Let $S$ be a closed densely defined symmetric operator in a Krein space $\bigl({\mathcal K},\kip \bigr)$.
As in the Hilbert space case, we define the  set of points of regular type of $S$ by, see \cite{AG},
\begin{equation*}
    \wh\rho(S) = \{ z \in \C : \text{exists } c_z  > 0 \text{ with } \| (S-z )x \| > c_z  \|x\| \text{ for all } x\in\mD(S)\}.
\end{equation*}

The \define{defect subspace} of the operator $S$ in a point of regular type $z$ is defined as follows
\begin{equation*}
\mathfrak{N}_z : = \ran(S - \overline{z})^{[\perp]}, \quad z \in\wh\rho(S).
\end{equation*}
The numbers $n_\pm(S):=\dim(\mathfrak{N}_z)$ are called the \define{defect numbers} 
of $S$ and they are constant for all $z\in\wh\rho(S)$ in the open upper (lower) half plane. 
Here $[\perp]$ stands for the orthogonal complement with respect to $\kip$.

In what follows we assume that the operator $S$ admits a selfadjoint extension $\wt {S}$ in $\bigl({\mathcal K},\kip\bigr)$ with a nonempty resolvent set $\rho(\wt {S})$. Then, for all $z\in\rho(\wt {S})$, we have the direct sum decomposition
\begin{equation}
   \label{eS*ds}
   \dom(S^{+}) = \dom (\wt {S}) \dotplus \mathfrak{N}_z.
\end{equation}
This implies that the dimension $\dim(\mathfrak{N}_z)$ is constant for all $z \in \rho(\wt {S})$ and consequently that $n_+(S)=n_-(S)$. In what follows, we consider only the case $n_+(S)=n_-(S)=1$.

\begin{definition}
\label{def:BTriple}
Let $\Gamma_0$ and $\Gamma_1$ be linear mappings from $\dom(S^+)$ to $\mathbb C$ such that
\begin{enumerate}[label={\upshape(\roman*)}]
   \item the mapping $\Gamma:\dom(S^+)\rightarrow\mathbb C^2$, 
   $f \mapsto
   \begin{pmatrix}
   	\Gamma_0f\\
   	\Gamma_1f
   \end{pmatrix}$,
   is surjective and
   \item the abstract Green's identity holds for all $f,~g\in \dom(S^+)$:
   \begin{equation}
   \label{Green}
   	\bigl[S^+f, g\bigr] - \bigl[f, S^+ g\bigr] = (\Gamma_1{f})\overline{{(\Gamma_0{g})}}\,-(\Gamma_0{f})\overline{{(\Gamma_1{g})}}.
   \end{equation}
\end{enumerate}
Then the triple $\bigl(\mathbb{C},\Gamma_0,\Gamma_1\bigr)$ is called a \define{boundary triple} for $S^+$, cf.\  \cite{D95,DM91,GG} for more generality.
\end{definition}

It follows from \eqref{Green} that the extensions $S_0$, $S_1$ of $S$ defined as restrictions of $S^+$ to the domains
\begin{equation*}
	\dom(S_0) :=  \ker(\Gamma_0)\quad \mbox{and} \quad \dom(S_1) :=\ker(\Gamma_1)
\end{equation*}
are selfadjoint extensions of $S$. Moreover, one can always choose  a boundary triple $\bigl(\mathbb{C},\Gamma_0,\Gamma_1\bigr)$ for $S$ such that $S_0=\wt S$, \cite[Proposition~2.2]{D99}. In this case  for every $z\in\rho(S_0)$ the decomposition~\eqref{eS*ds} holds with $\wt{S} = S_0$ and the mapping $\Gamma_0|_{\mathfrak{N}_z}:\mathfrak{N}_z\to\mathbb{C}$ is invertible for every $z\in\rho(S_0)$. 
This leads to the definition of the following functions.

\begin{definition}
\label{W00A}
The  functions $z\mapsto M(z)$ and $z\mapsto \gamma(z)$, defined as 
\begin{equation*} 
    M(z)\Gamma_0f_{z}=\Gamma_1f_{z}, \quad\text{and} \quad \gamma(z)=\left(\Gamma_0|_{\mathfrak{N}_z}\right)^{-1}, \quad f_z\in\mathfrak{N}_z,\,z\in\rho(S_0),
\end{equation*}
are called the \define{abstract Weyl function} of $S$ and the \define{$\gamma$-field} associated to the boundary triple $\bigl(\mathbb{C},\Gamma_0,\Gamma_1\bigr)$, respectively.
\end{definition}
The concept of the abstract Weyl function was introduced in~\cite{DM91} for symmetric operators in Hilbert spaces and later extended to symmetric operators in Krein spaces in~\cite{D95}.

\section{Operators on the star graph $G$} 

Let $G$ be a star graph with $n$ edges, each of unit length and joined at $0$. An element $f$ from the space $L^2(G)$ in  \eqref{eq:hilbert_graph} is of the form $f = (f_1, f_2, \dots, f_n)$ where $f_j \in L^2(I_j)$. The inner product on $L^2(G)$ is induced by the usual one on the intervals, that is,
\begin{equation*}
   \scalar{f}{g} = \sum_{j=1}^n \int_0^1 f_j(t) \overline{g_j(t)}\, \rd t,\qquad f, g\in   L^2(G).
\end{equation*}
Let $n_+\in \{1, \dots, n-1\}$ and $n_- := n- n_+$. 
On $L^2(G)$ we define the indefinite inner product
\begin{equation}
   \label{eq:kreinspace.scalar_product}
   \kscalar{f}{g} := \scalar{f}{Jg} = \sum_{j=1}^{n_+} \int_0^1 f_j(t) \overline{g_j(t)}\, \rd t -  \sum_{j=n_+ +1}^{n} \int_0^1 f_j(t) \overline{g_j(t)}\, \rd t
\end{equation}
whose fundamental symmetry $J: L^2(G) \to L^2(G)$ is given by
\begin{equation*}
   (Jf)_j := 
   \begin{cases}
      f_j,\quad & 1\le j \le n_+, \\[1ex]
      -f_j,\quad & n_+ + 1 \le j \le n,
   \end{cases}
\end{equation*}
which clearly satisfies $J^* = J^{-1} = J$ for $f,g \in L^2(G)$. The inner product in \eqref{eq:kreinspace.scalar_product} endows $(L^2(G),\kip)$ with a Krein space structure, see Section~\ref{sec:interval}.

Recall the definition of $\widehat H^1_D(G)$ in \eqref{eq:hilbert_graph.dirichlet}. 

In order to study the spectral properties of the operator $B$ in \eqref{DerAuserwaehlte} and \eqref{ElDominio}
we define auxiliary operators 
\begin{equation*}
   \SK \subseteq \AK \subseteq \TK
\end{equation*}
with domains 
\begin{align}
\label{eq:definition_minimal_operator}
   \mD(\TK ) &= \Big\{ f\in \widehat H^1_D(G) \ :\ f_1 = \dots = f_n(0),\,  
   f_j\in H^2(I_j) \text{ for all }  j\ \Big\},
   \\
   \mD(\AK) &= \{ f\in \mD(\TK) \ :\  f_j(0) = 0\ \text{ for all } j \},
   \\
   \mD(\SK ) &= \Big\{ f\in \mD(\TK) \ :\ f_j(0) = 0\  \text{ for all }  j\ \text{and}\ \sum_{j=1}^n f_j'(0) = 0 \Big\}
\end{align}
as
\begin{align*}
  \TK f & = (-f_1'',\, \dots, \, -f_{n_+}'',\,  f_{n_+ + 1}'',\, \dots, \,  f_n'')
\end{align*}
and
\begin{equation*}
    Af = Tf,\quad
    Sg = Tg,\qquad
    \text{for}\quad f\in\mD(A),\ g\in\mD(S).
\end{equation*}
For all three operators, the functions in their domains satisfy a Dirichlet boundary condition at the outer vertices of the graph and they are continuous at the central vertex. 
However, the additional condition at the central vertex is different for each of the operators.

Note that the functions in $\mD(\TK)$ are continuous on $G$, twice differentiable in the interior of the edges $I_j$ while their derivatives may be discontinuous in the central vertex $0$. 
The functions in the domain of $\AK$ are continuous on $G$ and they satisfy a Dirichlet condition in $0$.
Therefore, $A$ can be viewed as the direct sum of operators defined on the individual edges $I_j$, see Remark~\ref{rem:ST}.
Finally, the functions in the domain of $\SK$ satisfy in addition a Kirchhoff boundary condition at $0$.
\smallskip

For each $j= 1, \dots, n$, we set $\Aj_0$ to be the negative second derivative in the Hilbert space $L^2(I_j)$ on the edge $I_j$ with Dirichlet boundary conditions at the vertices $0$ and at $1$,
\begin{equation}
\label{eq:An}
	\DiffInt g = -g'',\quad \mD(\DiffInt) = \{ g\in H^2((0, 1)) : g(0) = g(1) = 0 \}.
\end{equation}
Then, $\DiffInt$ is selfadjoint and it has purely discrete spectrum
\begin{equation}
   \label{eq:Tn:spectrum}
   \sigma(\DiffInt) = \sigma_p(\DiffInt) = \{ (k\pi)^2 : k\in\N \}.
\end{equation}
All eigenvalues are simple and the corresponding eigenfunctions are
\begin{equation}
   \label{eq:uk}
   \eigf{k}(x) = \sin(k\pi x),\qquad x\in [0, 1].
\end{equation}

We collect the properties of $A$ in the following remark.
\begin{remark}
\label{rem:ST}

\begin{enumerate}[label={(\roman*)}]
\addtolength{\itemsep}{-.2\baselineskip}
   \item
   \label{item:ST:i} 
   The operator $\SK$ is a closed symmetric operator in the Krein space $(L^2(G), \kip)$ and $\SK^+ = \TK$.
   
   \item\label{item:ST:ii} 
   We have $\SK \subseteq \AK \subseteq \TK = \SK^+$.
   
   \item\label{item:ST:iii} 
   The operator $\AK$ is selfadjoint in the Krein space $(L^2(G), \kip)$ and
   \begin{align*}
   	\AK = 
   	\Big( \bigoplus_{j=1}^{n_+} \Aj_0 \Big) \oplus
   	\Big( \bigoplus_{j=n_+ + 1}^n -\Aj_0 \Big).
   \end{align*}
   Moreover, $A$ is a positive operator in $(L^2(G), \kip)$.
   Indeed, Poincar\'e's inequality shows that for all $f\in\mD(A)\setminus \{0\}$
   $$
    [Af,f]=   \scalar{JAf}{f} = \scalar{	\Big( \bigoplus_{j=1}^{n} \Aj_0 \Big) f}{f} 
      = \sum_{j=1}^n  \| f'_j \|_{L^2(I_j)}^2  
     = \| f' \|_{L^2(G)}^2 
      \ge \| f \|_{L^2(G)}^2. 
   $$
   
   \item\label{item:ST:iv} 
   The spectrum of $\AK$ consists of point spectrum only and 
   \begin{equation}
   \label{eq;pointspecA}
   	\sigma(\AK) = \pointspec(\AK) = \{ (k\pi)^2  : k\in\N \} \cup \{ -(k\pi)^2  : k\in\N \}.
   \end{equation}
   The eigenspace associated to each positive eigenvalue has dimension $n_+$ and the eigenspace associated to each negative eigenvalue has dimension $n_-$. In particular, $0 \in \rho(A)$.
\end{enumerate}
\end{remark}

Since any function $f\in \mD(\TK)$ is continuous at zero, the components $f_j$ satisfy
\begin{equation*}
   f_1(0) = \dots = f_n(0).
\end{equation*}
For convenience we write $f(0)$ for this common value.

\begin{proposition}
   \label{prop:boundarytriplet}
      The mappings
      \begin{alignat*}{3}
      \Gamma_0:~ &\mD(\TK) \to \C,\qquad &&\Gamma_0 f = f(0), \\
      \Gamma_1:~ &\mD(\TK) \to \C,\qquad &&\Gamma_1 f = \sum_{j=1}^n f_j'(0), 
   \end{alignat*}
   form a boundary triple for $\TK$.
\end{proposition}
\begin{proof}
   By definition, for $f,g\in \mD(\TK)$,  integration by parts, Dirichlet conditions at the outer vertices,  and the continuity of $f$ and $g$ at $0$ yields
   \begin{align*}
      \kscalar{\TK f}{g} - \kscalar{f}{\TK g} 
      &= \sum_{j=1}^{n} \scalar{-f_j''}{g_j}_{L^2(I_j)} - \scalar{f_j}{-g_j''}_{L^2(I_j)} 
      = \overline g(0) \sum_{j=1}^{n} f_j'(0) - f(0) \sum_{j=1}^{n} \overline g_j'(0)
      \\
      &= (\Gamma_1 f) (\overline{ \Gamma_0 g)} - (\Gamma_0 f) (\overline{ \Gamma_1 g}),
   \end{align*}
   hence $\bigl(\mathbb{C},\Gamma_0,\Gamma_1\bigr)$  is a boundary triple for $T$. 
\end{proof}

Observe that 
\begin{align}
   \mD(\AK) &= \ker \Gamma_0 = \{ f\in \mD(\TK) : \Gamma_0 f = 0 \},
   \\
   \mD(\AKone) &= \ker \Gamma_1 = \{ f\in \mD(\TK) : \Gamma_1 f = 0 \}
\end{align}
where $B$ is the operator from \eqref{DerAuserwaehlte}, \eqref{ElDominio}.
The next proposition describes the relation between the operator $B$ and the auxiliary operator $A$. 
This is the main result of this section.
\begin{proposition}
   The operator $B$ in \eqref{DerAuserwaehlte} is selfadjoint in the Krein space $(L^2(G), \kip)$. 
   That is,
   \begin{equation*}
      B=B^+.
   \end{equation*}
	For $\lambda\in\rho(A)\cap\rho(B)$, the following equality holds
    \begin{equation}
    \label{KREIN}
    	(A - \lambda)^{-1} = (B - \lambda)^{-1} - \gamma(\lambda)M(\lambda)^{-1}\gamma(\overline \lambda)^+,
    \end{equation}
    where $M$ is the Weyl-function and $\gamma$ is the $\gamma$-field, given by
    \begin{align}
    \label{eq:gammafield}
        \gammaK(\lambda)c &= 
   	\begin{cases}
   		- \frac{c \sin(\mu x - \mu)}{\sin\mu}\quad &\text{if }\ 1 \le j \le n_+,\\[1ex]
   		- \frac{c \sin(\I\mu x - \I\mu)}{\sin(\I\mu)} &\text{if }\ n_+ + 1 \le j \le n,
   	\end{cases}
    \\[1ex]
    \label{eq:Weylfunction}
    	\MK(\lambda) &= 
    	\begin{cases}
    		-\mu ( n_+\cot\mu + n_-\coth \mu ), \quad & \text{if }\ \lambda \neq 0,\\
    		- n,\quad & \text{if }\ \lambda = 0,
    	\end{cases}
    \end{align}
    for $\lambda\in\rho(\AK)$, $c\in \mathbb C$ and $\mu\in\C$ such that $\mu^2 = \lambda$.
\end{proposition}

\begin{proof}
   The selfadjointness of the operator $B$ and the formula~\eqref{KREIN} follow from \cite{D95}, see also Section~\ref{sec:interval}.
   In order to show \eqref{eq:Weylfunction} we fix $c\in\C$. Let us first assume that $\lambda\neq 0$. Then a straightforward calculation shows that $(\TK - \lambda)f = 0$ with $f = (f_j)_{j=1}^n$ and $\Gamma_0f = c$ if and only if 
   \begin{align*}
   	f_j(x) = 
   	\begin{cases}
   		- \frac{c \sin(\mu x - \mu)}{\sin\mu}\quad &\text{if }\ 1 \le j \le n_+,\\[1ex]
   		- \frac{c \sin(\I\mu x - \I\mu)}{\sin(\I\mu)} &\text{if }\ n_+ + 1 \le j \le n.
   	\end{cases}
   \end{align*}
   We obtain easily 
   \begin{equation*}
      \gammaK(\lambda) c =
      (\Gamma_0|_{\ker(T - \lambda)})^{-1} c = (f_j)_{j=1}^n.
   \end{equation*}
   Consequently,
   \begin{align*}
   	\MK(\lambda)c = \Gamma_1 (\Gamma_0|_{\ker(T-\lambda)})^{-1} c & = \sum_{j=1}^n f_j'(0) = \sum_{j=1}^{n_+} f_j'(0)+ \sum_{j = n_+ + 1}^n f_j'(0)\\
   	& = -\left( n_+ \frac{ \mu\cos\mu}{\sin\mu} + n_- \frac{ \I\mu \cos(\I\mu)}{\sin(\I\mu)}\right) c.
   \end{align*}
   The case $\lambda = 0$ is treated analogously.
\end{proof}

\section{Spectral properties of the indefinite Kirchhoff Laplacian 
} \label{spectrum}

We now establish an interlacing property for the eigenvalues of $A$ and $B$. 
The spectrum of $A$ is pure point spectrum.
Since $B$ is a one-dimensional perturbation of $A$, see \eqref{KREIN},
the dimensions of their eigenspaces for the same $\lambda$ differ by at most $1$.
Indeed, we will show that 
$\dim(\ker(B-\lambda) = \dim(\ker(A-\lambda) - 1$ for every$\lambda\in \pointspec(A)$ and that between any two subsequent eigenvalues of $A$ there is exactly one eigenvalue of $B$ and this eigenvalue has multiplicity $1$.

\subsection{The spectrum of the operator $B$}
\begin{theorem}
\label{thm:spectrumB}
   The operator $B$ is a positive operator in the Krein space $(L^2(G), \kip)$, $0\in \rho(B)$, the spectrum of $\AKone$ is purely discrete, real and 
   \begin{align*}
   	\sigma(\AKone) &=  \sigma_p(\AKone) =
        \begin{cases}
        \{ -(k\pi )^2: k\in\N \}\cup \{ \eta_{k} : k\in\mathbb Z \setminus\{0\} \},
        &\text{if }\ n_+=1,
        \\[1ex]
        \{ (k\pi )^2: k\in\N \}\cup \{ \eta_{k} : k\in\mathbb Z \setminus\{0\} \},
        &\text{if }\ n_-=1,
        \\[1ex]
        \{ \sign(k)(k\pi )^2: k\in\Z\setminus\{0\} \}\cup \{ \eta_{k} : k\in\mathbb Z \setminus\{0\} \},
        &\text{if }\ n_+\ge 2, n_-\ge 2
        \end{cases}
   \end{align*}
   where $\eta_k$ is the unique solution of 
   \begin{enumerate}[label={\upshape(\roman*)}]
   	\item $\coth(\sqrt{\eta_k}) = - \frac{n_+}{n_-}\cot(\sqrt{\eta_k})$ in the interval $( ((k-1)\pi)^2, (k\pi)^2 )$ if $k\ge 1$,
   	\item $\coth(\sqrt{-\eta_k}) = - \frac{n_-}{n_+}\cot(\sqrt{-\eta_k})$ in the interval $( -(k\pi)^2, -((k+1)\pi)^2 )$ if $k\le -1$.
   \end{enumerate}
   In particular, 
   \begin{equation*}
      -\pi^2 < \eta_{-1} <  0 < \eta_1 < \pi^2
   \end{equation*}
   and
   \begin{alignat*}{3}
      -(k\pi)^2 &< \eta_k < -((k+1)\pi)^2 
      && \quad\text{for } k\le -2\\
      ((k-1)\pi)^2 &< \eta_k < (k\pi)^2 
      && \quad\text{for } k\ge 2.
   \end{alignat*}
   The multiplicity of $(k\pi)^2$ as an eigenvalue of $\AKone$ is $n_{+} - 1$, while the multiplicity of $-(k\pi)^2$ is $n_{-} - 1$. Each $\eta_k$ has multiplicity $1$. Moreover, the eigenvalues $\eta_k$ satisfy the asymptotic behavior
   
   \begin{equation*}
   	\eta_{k} \approx 
   	\begin{cases}
   		\left( k\pi - \arctan \frac{n_-}{n_+} \right)^2 \quad & \text{for }\ k\to\infty,\\[1ex]
   		-\left( k\pi - \arctan \frac{n_+}{n_-} \right)^2 \quad
   		&\text{for }\ k\to -\infty.
   	\end{cases}
   \end{equation*}
   Therefore, the spectrum of $B$ is symmetric with respect to the origin if and only if
   $$
   n_+=n_-.
   $$
\end{theorem}

\begin{proof}
   The Weyl function $M$ (cf.\ Section~\ref{sec:interval}) contains information about the spectrum of $\AKone$.
   The spectrum of $\AKone$ consists of the zeros of $M$ and possibly a subset of the spectrum of $\AK$.
    This follows from \eqref{KREIN}: For $\lambda \in \rho(A)$ the left hand side of \eqref{KREIN} is bounded, hence, $\lambda \in \sigma(B)$ if and only if $M(\lambda)=0$ 
   (see, e.g., \cite[Theorem 2.1]{D95}), 
   \begin{equation}
      \sigma(\AKone)\subseteq \sigma(\AK) \cup \{ \lambda\in\C : \MK(\lambda) = 0 \}.
   \end{equation}
   For $f\in\mD(B)$ we have
   \begin{align}
   \label{Bananenwurst}
      \scalar{J\AKone f}{f}
      & = \sum_{j=1}^n \int_0^1 f_j' \overline{f_j'}\, dx + f_j' \overline f_j\Big|_0^1 = \| f' \|_{L^2(G)}^2 +  \overline f_1(0)\, \sum_{j=1}^n f_j'(0) = \| f' \|_{L^2(G)}^2 
      \ge 
      \| f \|_{L^2(G)}^2 
      >0
   \end{align}
   where in the last step we used Poincar\'e's inequality.
   Therefore $JB$ is an uniformly positive operator in the Hilbert space $(L^2(G),\scalar{\cdot}{\cdot})$ and $0\in\rho(JB)$. As $J=J^{-1}$, we obtain  $0\in\rho(B)$.
   Hence the operator $B$ is a positive operator in the Krein space $(L^2(G), \kip)$, implying 
   (see Section \ref{sec:interval})
   \begin{equation*}
   	\sigma(\AKone)\subseteq \R.
   \end{equation*}
   It follows from~\eqref{eq:Weylfunction} that $\MK(0)\neq0$. In order to find all positive zeros of $\MK$, we substitute $\mu = \sqrt{\lambda} > 0$ in \eqref{eq:Weylfunction}. Then
   \begin{equation*}
   	0 = \MK(\lambda) = -\mu (n_+ \cot\mu + n_- \coth \mu).
   \end{equation*}
   Equivalently,
   \begin{equation}
   \label{eq:eigenvalue_positive}
       \frac{n_-}{n_+} \coth \mu =	- \cot\mu.
   \end{equation}
   Note that $\coth$ is strictly decreasing from $\infty$ to $1$ in the interval $(0, \infty)$. 
   For each $k\ge 1$, $-\cot\mu$ is continuous and strictly increasing from $-\infty$ to $\infty$ on the interval $( (k-1)\pi,\, k\pi)$. 
   Therefore, for any $k\ge 1$, the equation \eqref{eq:eigenvalue_positive} has a unique solution $\mu_k$ in the interval $((k-1)\pi,\, k\pi)$, see Figure~\ref{fig:neweigenvalue}.

   To see the asymptotic behaviour, recall that 
   $\tanh \mu\to 1$ for $\mu \to \infty$. 
   Hence, $\mu$ is asymptotically the unique solution of $\tan \mu = -\frac{n_+}{n_-}$ as $k\rightarrow\infty$ in the interval $((k-1)\pi,\, k\pi)$ as $k\to\infty$.

   For finding the negative zeros of $\MK$, we substitute $\mu = \I\sqrt{|\lambda|}$ in \eqref{eq:Weylfunction}. This gives
   \begin{equation*}
   	0 = \MK(\lambda) = -\mu (n_+ \cot(\I |\mu| ) + n_- \coth (\I |\mu|)) = \I\mu (n_+ \coth(|\mu| ) + n_- \cot (|\mu|)),
   \end{equation*}
   which is equivalent to
   \begin{equation*}
   	\cot|\mu| = - \frac{n_+}{n_-} \coth |\mu|.
   \end{equation*}
   With an analogous argument to the positive root case, we conclude that for every $k\le -1$ the interval $( k\pi,\, (k+1)\pi)$ contains exactly one solution $\mu_k$ of the above equation. Similarly, $\mu_k\approx k\pi - \arctan \frac{n_+}{n_-}$ for $k\to\infty$.
   
   We prove the claim about the multiplicities of the eigenvalues of $\AKone$. Let $I\subseteq \R$ be a finite interval. Then we denote by $m_{A}(I)$ (resp. $m_{B}(I)$) the sum of the algebraic multiplicities of the eigenvalues of $A$ (resp.\ $B$) which are contained in $I$. We know from \cite{BLMMT2016} 
   that for $0\notin I$ and
   \begin{equation}
   \label{eq:multiplicities}
   	m_{\AK}(I) - 1 \le  m_{\AKone}(I) \le  m_{\AK}(I) + 1.
   \end{equation}
   Choosing $I=( ((k-1)\pi )^2,\, (k\pi)^2)$ for $k\ge 1$ gives
   \begin{equation*}
   	m_{\AK}(I) =0\quad \mbox{and }\quad m_{\AKone}(I) = m_{\AKone}(\{\eta_k\}) \ge 1
   \end{equation*}
   which together with \eqref{eq:multiplicities} gives 
   \begin{equation*}
   	m_{\AKone}(\{\eta_k\}) = 1.
   \end{equation*}
   On the other hand, choosing $I=( ((k-1)\pi )^2,\, ( (k+1)\pi)^2)$ for $k\in\N$ gives
   \begin{equation*}
   	m_{\AK}(I) = m_{\AK}(\{(k\pi)^2\}) = n_+
   \end{equation*}
   and
   \begin{equation*}
   	m_{\AKone}(I) = m_{\AKone}(\{(k\pi)^2\}) + m_{\AKone}(\{\eta_{k-1}\})  + m_{\AKone}(\{\eta_k\}) = m_{\AKone}(\{(k\pi)^2\}) + 2.
   \end{equation*}
   Again, together with \eqref{eq:multiplicities},
   \begin{equation*}
   	m_{\AKone}((k\pi)^2) \le n_+ - 1.
   \end{equation*}
   Finally, we take a small neighborhood $I$ around $(k\pi)^2$ such that $\eta_k$ and $\eta_{k+1}$ are not in $I$. By \eqref{eq:multiplicities},
   \begin{equation*}
   	m_{\AKone}(\{(k\pi)^2\})=m_B(I)\geq   m_{\AK}(I) - 1 = n_+ - 1.
   \end{equation*}
The statement for the negative eigenvalues of $B$ follows analogously.
\end{proof}
The above theorem gives now the opportunity to deduce from the spectrum of  the operator $B$ the topology of the graph $G$.
\begin{corollary}
 Let $\eta_1$ be the first positive eigenvalue of the operator $B$. Then the ratio between $n_+$ and $n_-$ is given as
 $$
 \frac{n_+}{n_-}=-\coth(\sqrt{\eta_1}) \tan(\sqrt{\eta_1}).
 $$
\end{corollary}

\begin{remark*}
   \begin{enumerate}[label={(\roman*)}]
   \item  
   If $n_+ = n_- = n/2$, then 
   $\eta_k = \sign(k) ( ( |k| - \frac{1}{4})\pi )^2$ for all $k \in\mathbb Z\setminus\{0\}$ and $\sigma(\AKone)$ is symmetric with respect to $0$.

    \item If $\frac{n_-}{n_+}\to 0$, then the eigenvalues $\eta_k$ of $\AKone$ move to the right and tend to $(k\pi)^2$ (the $k$th positive eigenvalue of $\AK$) if $k> 0$ and to $-( (|k| - \frac{1}{2} )\pi)^2$ if $k<0$.
    Therefore the positive spectrum of $\AKone$ resembles the positive spectrum of $\AK$ while for the negative eigenvalues $\eta_k$ stay away from the eigenvalues of $\AK$.
   \end{enumerate}
\end{remark*}

\begin{figure}
   \begin{center}
      \begin{tikzpicture}[scale=1, transform shape]
	 \begin{axis}[ymin=-3, ymax=4, xmin=-0.015, xmax=15,
	    axis x line=center,
	    axis y line=center,
	    xlabel={$\mu$},
	    xlabel style={right},
	    xtick={pi,2*pi, 3*pi, 4*pi},
	    xticklabels={$\pi$, $2\pi$, $3\pi$, $4\pi$},
	    ylabel style={left},
	    ytick={1},
	    yticklabels={$\frac{n_+}{n_-}$},
	    legend style={
	    anchor=east,legend columns=3},
	    legend style={font=\tiny},
	    declare function={ 
	    lhs(\t) = (tanh( \t ))^(-1);
	    rhs(\t) = -cot(\t r);
	    tanrhs(\t) = tan(\t r);
	    }
	    ]
	    \pgfplotsset{
	    tick label style={ font=\footnotesize},
	    }
	    \addplot[blue, domain=0:15, samples=800, restrict y to domain=-20:20]{rhs(x)};
	    \addplot[red, domain=0:15, samples=500
	    ]{lhs(x)};
	 \end{axis}
      \end{tikzpicture}
   \end{center}
   \caption{Intersections of $-\cot\mu$ and $\frac{n_+}{n_-}\coth\mu$ .}
   \label{fig:neweigenvalue}
\end{figure}
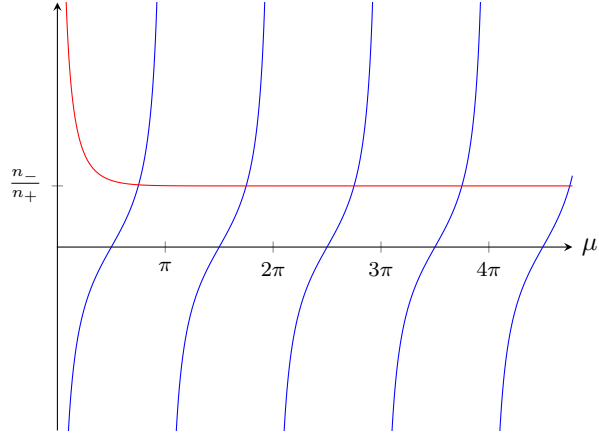

\subsection{Similarity of \texorpdfstring{$\AKone$}{B} to a selfadjoint operator on a Hilbert space}
\label{subsection:similarity}

	For stating and proving the main result in this section, we start with two lemmas. Lemma \ref{lemma:invertible_operator} is adapted from Lemma~3.2 from \cite{CurgusLanger1989} and we compute the form domain of $JB$ in Lemma \ref{lemma:form_domain}.

\begin{lemma}
\label{lemma:invertible_operator}
   For every $\delta\in (0, 1/2)$ there exists a positive bounded and boundedly invertible operator $X: L^2(G)\to L^2(G)$ in the Hilbert space $(L^2(G),\scalar{\cdot}{\cdot})$ such that
   for every $f = (f_j)_{j=1}^n\in L^2(G)$ the following is true.
   \begin{enumerate}[label={(\roman*)}]
	\item\label{item:Xi} 
    $(f-Xf)_j(x) = 0$ for $x \in(1-\delta, 1)\subseteq I_j$;
    
	\item\label{item:Xii} 
        if $f_j$ has a continuous extension to $0$, then so does $(Xf)_j$ and $(Xf)_j(0) = 0$.
   \end{enumerate}
\end{lemma}

\begin{proof}
   Let $0<\delta<1/2$ and choose $\varphi\in C^1([0,1])$, $0\leq\varphi\leq1$, which is equal to $1$ on $[0,\delta/2)$ and vanishes in $[\delta,1]$. Define the linear operator $Y=Y_1\oplus\cdots\oplus Y_n$ in $L^2(G)$ as
   \begin{equation}
      \label{eq:axillary_operator_Y}
      Y_j u_j(x) = 
      \begin{cases}
	 \left(\alpha_j s_j u_j(s_j x) + \beta_j t_j u_j(t_j x)\right)\varphi(x) \quad & \text{for }\ x\in [0,\delta], \\[1ex]
	 0\quad & \text{for }\ x\in (\delta,1],
      \end{cases}
   \end{equation}
   where  $s_j ,t_j \in (1,\, 2)$, $s_j \neq t_j$ and $\alpha_j, \beta_j \in \mathbb R$ for $j=1,\dots, n$.  
   Below in \eqref{eq:alphabeta} we will determine an appropriate choice for the coefficients $\alpha_j$ and $\beta_j$.
   It is easy to see that $Y$ is bounded and that for all $j=1, \dots, n$
   \begin{align}
      \label{eq:axillary_operator_Ystar}
      ( Y^*u)_j (x) = Y_j^*u_j (x) 
      =  
      \alpha_j u_j\left(\frac{x}{s_j}\right)\varphi \left(\frac{x}{s_j}\right) 
      + \beta_j u_j\left(\frac{x}{t_j}\right)\varphi \left(\frac{x}{t_j}\right),
   \end{align}
   Moreover, if $u_j$ admits a continuous extension at $0$, then so do $(Yu)_j$ and $(Y^*u)_j$. Note that by construction, $(Yu)_j$ and $(Y^*u)_j$ are zero in $[\delta, 1)\subseteq I_j$ for each $j=1, \dots, n$ and $u\in L^2(G)$. 
   Next we want to choose the coefficients $\alpha_j$ and $\beta_j$ such that 
   \begin{equation}
      \label{eq:alphabeta}
      (Yu)_j(0) = u_j(0) \qquad\text{and}\qquad  (Y^*u)_j(0) = -u_j(0)
   \end{equation}
   whenever $u_j$ has a continuous extension to $0$. By construction of $Y$, this is clearly equivalent to
   \begin{equation}
      \alpha_j s_j + \beta_j t_j  = 1, \qquad \alpha_j + \beta_j  = -1, \qquad\qquad j = 1, \dots, n.
   \end{equation}
   This system of linear equations has the unique solution
   \begin{equation*}
      \alpha_j =  \frac{1+t_j}{s_j- t_j} \quad \mbox{and }\quad\beta_j =  -\frac{1+s_j}{s_j- t_j}.
   \end{equation*}
   From all the above it follows that the operator $X := 1 + Y^*Y$ has the desired properties.
\end{proof}

The next lemma describes the form domain of $J\AKone$. Since $J\AKone$ is a positive operator in the Hilbert space $H$ by \eqref{Bananenwurst}, its form domain $\mD[J\AKone]$ is defined and, according to the Second Representation Theorem, it is equal to $\mD\left((J\AKone)^{1/2}\right)$ (see, e.g., \cite[VI-§ 2.6]{K}). The following lemma gives a more explicit formulation and a direct computation.

\begin{lemma}
   \label{lemma:form_domain}
   The form domain $\mD[J\AKone]$ of $J\AKone$ is equal to
   \begin{align}
      \label{eq:formdomainAone}
      \mD[J\AKone] = \left\{ f\in \widehat H^1_D(G)  : f_1(0) = \dots = f_n(0) \right\}.
   \end{align}
\end{lemma}

\begin{proof}
   The claim can be shown with the general results from \cite[Theorem 1.4.11]{BerkolaikoKuchment}, but we give a direct proof as our setup is relatively simple. It follows from  \eqref{Bananenwurst} that $\mD[J\AKone]$ is the closure of $\mD(J\AKone)$ with respect to the norm
   \begin{equation*}
      \|f\|_1 := \|f\|_{L^2(G)} + \|f'\|_{L^2(G)}
   \end{equation*}
   (see, e.g., \cite[VI-§ 2.1]{K}). Let us denote the right hand side of \eqref{eq:formdomainAone} by $\mD$. We first show $\mD[J\AKone]\subseteq \mD$. Let us fix $f\in \mD[J\AKone]$ and choose a sequence $(g_t)_{t\in\N}\subseteq \mD(J\AKone)$ such that $g_t\rightarrow f$ with respect to $\|\cdot\|_1$, Then $(g_{t,j})_{t\in\N}$ converges to $f_j$ in $H^1(I_j)$ on each edge $I_j$. Since convergence in $H^1(I_j)$ implies pointwise convergence of $g_{t,j}$ to $f_j$ in the endpoints $0$ and in $1$ on each edge \cite[Lemma 1.3.8.]{BerkolaikoKuchment}, it follows that
   \begin{equation*}
      f_j(1) = \lim_{t\to\infty} g_{t,j}(1) = 0\quad\mbox{and } \quad f_j(0) - f_k(0) = \lim_{t\to\infty} ( g_{t,j}(0) - g_{t,k}(0)) = 0
   \end{equation*}
   for all $j,k= 1,\dots, n$. Hence $f\in \mD$.

   Now we show the converse inclusion $\mD\subseteq \mD[J\AKone]$. To this end we set 
   \begin{equation*}
      \mD_0 := \bigoplus_{j=1}^n H_0^1(I_j).
   \end{equation*}
   Since $H^1_0(I_j)$ is the closure of  $\mathcal C_0^\infty(I_j)$ with  respect to the norm $\|\cdot\|_1$, we have
   \begin{equation*}
      \mD_0 = \overline{\bigoplus_{j=1}^n \mathcal C_0^\infty(I_j)}^{\|\cdot\|_1} \subset\ \overline{\mD(J\AKone)}^{\|\cdot\|_1}= \mD[J\AKone].
   \end{equation*}
   Fix an infinitely often differentiable function 
   $\phi$ defined on $G$
   such that for all $j=1,\dots, n$
   \begin{equation*}
      \phi_j(x) = 1\quad\text{if }\ x\in [0,1/4]
      \qquad\text{and}\qquad
      \phi_j(1) = 0.
   \end{equation*}
   Then clearly $\mD = \mD_0 + \linspan\{ \phi\}$.
   Moreover, $\phi\in \mD(J\AKone)\subseteq \mD[J\AKone]$.
   Hence  
   \begin{equation*}
      \mD = \mD_0 + \linspan\{ \phi\}\subseteq \mD[J\AKone].
      \qedhere
   \end{equation*}
\end{proof}

The next theorem is the main result.
\begin{theorem}
   The operator $\AKone$ is similar to a selfadjoint operator in the Hilbert space $L^2(G)$. In particular, its eigenfunctions form a Riesz basis.
\end{theorem}
\begin{proof}
   Since $\AKone$ is nonnegative with $0\in \rho(\AKone)$ it remains to show that there exists a positive, bounded and boundedly invertible operator 

   $$
   W:L^2(G) \to L^2(G) \quad \mbox{such that} \quad 
   W(\mD(( J\AKone )^{\frac{1}{2}}))) \subseteq 
   \mD(( J\AKone )^{\frac{1}{2}}),
   $$
   cf.\ Theorem~\ref{Goettingen}.
   We claim that the operator $W = JX$ with $X$ from Lemma~\ref{lemma:invertible_operator} is such an operator.
   Clearly, $X$ is bounded and boundedly invertible in $L^2(G)$, hence $W$ is bounded and boundedly invertible.
   Moreover, the positivity of $X$ in the Hilbert space $(L^2(G),\scalar{\cdot}{\cdot})$ gives the positivity of $W$ in the Krein space $(L^2(G),\kip)$.
   By construction, 
   $$
   Xf\in \bigoplus_{j=1}^n H^1(I_j) \quad 
   \mbox{if} \quad f\in \bigoplus_{j=1}^n H^1(I_j),
   $$
   hence by Lemma \ref{lemma:form_domain}
   $$
   X(\mD[J\AKone])\subseteq X\Big(\bigoplus_{j=1}^n H^1(I_j)\Big)\subseteq \bigoplus_{j=1}^n H^1(I_j).
   $$
   If $f\in \mD( (J\AKone)^{\frac{1}{2}}) )$, then $f_j(1) = 0$ for all $j=1,\dots, n$, 
   hence \ref{item:Xi} from Lemma~\ref{lemma:invertible_operator} shows that 
   $$
   f_1(1) = \dots = f_n(1) = 0
   $$
   and \ref{item:Xii} from the same lemma gives 
   $$
   f_1(0) = \dots = f_n(0) = 0.
   $$
   In summary, 
   $$
   Xf\in \bigoplus_{j=1}^n H_0^1(I_j).
   $$
   Clearly 
   $$
   J \Big( \bigoplus_{j=1}^n H_0^1(I_j) \Big) = \bigoplus_{j=1}^n H_0^1(I_j) \subseteq \mD[J\AKone].
   $$
   Therefore
   \begin{equation*}
      W(\mD[J\AKone]) = JX(\mD[J\AKone])
      \subseteq J \Big( \bigoplus_{j=1}^n H_0^1(I_j) \Big)
      \subseteq \mD[J\AKone].
      \qedhere
   \end{equation*}
\end{proof}

\bibliography{biblio}
\bibliographystyle{acm}
\noindent
Edison Leguizamon, {\tt ej.leguizamon@uniandes.edu.co}\\
Monika Winklmeier, {\tt mwinklme@uniandes.edu.co}\\
Departamento de Matem\'aticas\\
Universidad de Los Andes\\
Cra.\ 1a No 18A-70, 111711 Bogotá,
Colombia

\ \\
Carsten Trunk, {\tt carsten.trunk@tu-ilmenau.de}\\
Mitsuru Wilson, {\tt mitsuru.wilson@tu-ilmenau.de}\\
Ilmenau University of Technology\\
	PF 10 05 65, D--98684 Ilmenau, Germany

\end{document}